\documentclass{amsart}
\date{November 06, 2010}
\usepackage{verbatim,enumerate}
\title[$A_2$-singularities of hypersurfaces]{
  $A_2$-singularities of hypersurfaces
  with non-negative sectional curvature \\
  in Euclidean space}
\author{Kentaro Saji}
\address[Saji]{%
  Department of Mathematics,
  Faculty of Education,
  Gifu University,  Yanagido 1-1, Gifu 501-1193, Japan
}
\email{ksaji@gifu-u.ac.jp}

\author{Masaaki Umehara}
\address[Umehara]{%
   Department of Mathematics, Graduate School of Science,
   Osaka University,
   Toyonaka, Osaka 560-0043,
   Japan
}
\email{umehara@math.sci.osaka-u.ac.jp}
\author{Kotaro Yamada}
\address[Yamada]{%
   Department of Mathematics,
   Tokyo Institute of Technology,
   O-okayama, Meguro, Tokyo 152-8551, Japan%
}
\email{kotaro@math.titech.ac.jp}

\subjclass[2000]{%
 Primary 57R45;   
 Secondary 53A05. 
}
\thanks{%
 K. Saji, M. Umehara and K. Yamada were partially 
 supported by Grant-in-Aid for Scientific Research 
 (Young Scientists (B)) No. 20740028,
 (A) No.22244006 and (B) No. 21340016,
 respectively from the Japan Society for the Promotion of Science.
}
\keywords{wave front, singular curvature, realization}
\renewcommand{\keywords}[1]{\ignorespaces\relax\ignorespaces}
\newcommand{\subclass}[1]{\ignorespaces\relax\ignorespaces}
\usepackage{amsthm}
\theoremstyle{plain}
 \newtheorem{theorem}{Theorem}[section]
 \newtheorem{proposition}[theorem]{Proposition}
 \newtheorem{fact}[theorem]{Fact}
 \newtheorem{lemma}[theorem]{Lemma}
 \newtheorem{corollary}[theorem]{Corollary}
\theoremstyle{definition}
 \newtheorem{definition}[theorem]{Definition}
\theoremstyle{remark}
 \newtheorem{remark}[theorem]{Remark}
 \newtheorem*{remark*}{Remark}
 \newtheorem{example}[theorem]{Example}
\numberwithin{equation}{section}

\newenvironment{enum}{%
  \begin{enumerate}\setlength{\itemindent}{1em}%
                   \setlength{\leftmargin}{20em}%
}{\end{enumerate}}
\newcommand{\vect}[1]{\boldsymbol{#1}}
\newcommand{\R}{\boldsymbol{R}}
\newcommand{\C}{\boldsymbol{C}}
\renewcommand{\phi}{\varphi}
\newcommand{\sgn}{\operatorname{sgn}}
\newcommand{\ext}{\operatorname{ext}}
\newcommand{\inner}[2]{\left\langle{#1},{#2}\right\rangle}
\newcommand{\Ker}{\operatorname{Ker}}

\newcommand{\SO}{\operatorname{SO}}
\renewcommand{\O}{\operatorname{O}}
\newcommand{\so}{\operatorname{\mathfrak{so}}}
\newcommand{\E}{\mathcal{E}}

\newcommand{\F}{\mathcal{F}}

\newcommand{\pmt}[1]{{\begin{pmatrix} #1  \end{pmatrix}}}

\newcommand{\trans}[1]{\vphantom{#1}^t\!#1}
\newcommand{\first}{\operatorname{\mathit{I}}}
\newcommand{\second}{\operatorname{\mathit{I\!I}}}
\newcommand{\third}{\operatorname{\mathit{I\!I\!I}}}

\begin{document}
\maketitle
\begin{abstract}
In a previous work, the authors
gave a definition of `front bundles'.
Using this, we give a realization 
theorem for wave fronts in 
space forms, like as in the fundamental 
theorem of surface theory.
As an application, we investigate the behavior
of principal singular curvatures
along  $A_2$-singularities of hypersurfaces
with non-negative sectional curvature in 
Euclidean space.
\end{abstract}

\setcounter{section}{-1}
\section{Introduction}
It is known that two Gauss-Bonnet formulas 
hold for compact orientable fronts (wave fronts) 
in $\R^3$ 
(see \cite{K1}, \cite{SUY1} and \cite{SUY_JGEA}). 
From this, it is expected that there is an intrinsic formulation of
wave fronts, as well as of their realization problem, like as in the
fundamental theorem of surface theory.

In this paper, we recall the definitions of
{\it coherent tangent bundles\/} and 
{\it front bundles\/} given in \cite{SUY_JGEA}, 
which is an intrinsic formulation for wave fronts,
and give a necessary and sufficient condition for
a given front bundle to  be realized as a wave front
in a space form (cf.\ Theorem~\ref{thm:fundamental}). 
As an application, we also give a necessary 
and sufficient condition
for a given coherent tangent bundle over a
manifold to be realized as a smooth map into a
same dimensional space form
(cf.\ Theorem~\ref{thm:morin}).

Moreover, using this new framework, we show the
following assertion, which is a generalization
of \cite[Theorem 3.1]{SUY1} for $2$-dimensional fronts.

\begin{theorem}
 Let $M^m$ be an $m$-manifold and
 $f:M^m\to \R^{m+1}$ a wave front
 with the singular set $\Sigma_f$. 
 Take an open subset $U(\subset M^m)$
 such that $U\cap \Sigma_f$ consists only of
 $A_2$-singular points.
 Then the following hold{\rm:}
\begin{enumerate}
 \item\label{item:main:1} 
      If the sectional curvature $K$ of
      the induced metric is bounded
      on $U\setminus \Sigma_f$, then the second fundamental
      form of $f$ vanishes along $\Sigma_f\cap U$.
 \item\label{item:main:2}
      If $K$ is non-negative
      on $U\setminus \Sigma_f$, then it is bounded 
      and the singular principal curvatures of $f$
      {\rm(}cf.\ Definition \ref{def:singular-curvature}{\rm)}
      along $U\cap \Sigma_f$ are all non-positive.
\end{enumerate}
\end{theorem}
The first assertion of \cite[Theorem 5.1]{SUY1}
is the same statement as \ref{item:main:1}.
This theorem follows from the corresponding intrinsic version of
the statements given in Theorem \ref{thm:bdd}, 
which enable us to prove the similar assertions 
for wave fronts in the space form of constant curvature $c$
by a suitable modification.
As a direct consequence of the theorem, 
we get the following assertion, which is the second
assertion of \cite[Theorem 5.1]{SUY1}.
\begin{corollary}
 \label{cor:hyper}
 Let $f\colon{}U\to \R^{m+1}$ $(m\ge 3)$ be a front whose 
 singular points are all $A_2$ points. 
 If the sectional curvature $K$ 
 is positive everywhere on the set of
 regular set points,
 the sectional curvature of the singular submanifold is non-negative.
 Furthermore, if $K\geq \delta(>0)$, then
 the sectional curvature of the singular submanifold is positive.
\end{corollary}

An example satisfying the condition in the theorem
and the corollary is given in \cite{SUY1}.
In this paper, we shall also give a new such example.

\section{Coherent tangent bundles}
\label{sec:coherent}
\subsection{Coherent tangent bundles and their singularities}
According to \cite{SUY_JGEA}, 
we recall a general setting for intrinsic fronts:
Let $M^m$ be an oriented $m$-manifold ($m\ge 1$).
A {\em coherent tangent bundle\/} over $M^m$ 
is a $5$-tuple $(M^m,\E,\inner{~}{~},D,\varphi)$, 
where 
\begin{enum}
 \item $\E$ is a vector bundle of rank $m$ over $M^m$ with
       an inner product $\inner{~}{~}$,
 \item $D$ is a metric connection on $(\E,\inner{~}{~})$,
 \item $\varphi\colon{}TM^m\to \E$ is a bundle homomorphism
       which satisfies
       \begin{equation}\label{eq:c}
             D^{}_{X}\phi(Y)-D^{}_{Y}\phi(X)-\phi([X,Y])=0
       \end{equation}
       for vector fields $X$ and $Y$ on $M^m$.
\end{enum}
In this setting, the pull-back of the metric 
\begin{equation}\label{eq:phi-metric}
   ds^2_\phi:=\phi^*\inner{~}{~}
\end{equation}
is called the {\em $\phi$-metric}, 
which is a positive semidefinite symmetric tensor on $M^m$.
A point $p\in M^m$ is called a {\it $\phi$-singular point\/} 
if $\phi_p\colon{}T_pM^m\to\E_p$ is not a bijection, where $\E_p$ is the
fiber of $\E$ at $p$, 
that is,  $ds^2_\phi$ is not positive definite at $p$.
We denote by $\Sigma_{\phi}$ the set of $\phi$-singular points on $M^m$.
On the other hand, a point $p\in M^m\setminus\Sigma_{\phi}$
is called a {\it $\phi$-regular point}.
By \eqref{eq:c}, the pull-back connection 
of $D$ by $\phi$ coincides with the Levi-Civita connection
with respect to $ds^2_\phi$ on the set of $\phi$-regular points.
Thus, one can recognize that the concept of coherent tangent bundles   
is a generalization of Riemannian manifolds.

A coherent tangent bundle $(M^m,\E,\inner{~}{~},D,\phi)$
is called {\it co-orientable\/} if the vector bundle $\E$ is orientable,
namely, there exists a smooth non-vanishing 
section $\mu$ of the determinant bundle of the dual bundle $\E^*$
such that 
\begin{equation}\label{eq:mu}
   \mu(\vect{e}_1,\dots,\vect{e}_m)=\pm 1
\end{equation} 
for any orthonormal frame
$\{\vect{e}_1,\dots,\vect{e}_m\}$ on $\E$.
The form $\mu$ is determined uniquely up to a $\pm$-ambiguity.
A {\em co-orientation\/} of the coherent tangent bundle $\E$ is a 
choice of $\mu$.
An orthonormal frame $\{\vect{e}_1,\dots,\vect{e}_m\}$ is called 
{\em positive\/} with respect to the co-orientation $\mu$ 
if $\mu(\vect{e}_1,\dots,\vect{e}_m)=+1$.

We give here typical examples of coherent tangent bundles:
\begin{example}[\cite{SUY_JGEA}]\label{ex:map}
 Let  $M^m$ be an oriented $m$-manifold
 and  $(N^m,g)$ an oriented Riemannian $m$-manifold.
 A $C^\infty$-map $f:M^m\to N^m$ induces a coherent tangent bundle
 over $M^m$  as follows: 
 Let $\E_f:=f^*TN^m$ be the pull-back of the tangent bundle $TN^m$
 by $f$.
 Then $g$ induces a positive definite metric $\inner{~}{~}$ on $\E_f$,
 and the restriction $D$ of the Levi-Civita connection of $g$
 gives a connection on $\E$ which is compatible with respect to
 the metric $\inner{~}{~}$.
 We set 
 $\phi^{}_f:=df:TM^m\to\E_f$,
 which gives the structure of the coherent tangent bundle on $M^m$.
 A necessary and sufficient condition for  a given coherent tangent
 bundle over an $m$-manifold  to be realized as
 a smooth map into an $m$-dimensional space form
 will be given in Theorem~\ref{thm:morin} in Section~\ref{sec:realization}.
\end{example}

\begin{example}[\cite{SUY_JGEA}]\label{ex:front}
 Let $(N^{m+1},g)$ be an $(m+1)$-dimensional Riemannian manifold.
 A $C^\infty$-map
 $f:M^m\to N^{m+1}$
 is called a {\it frontal\/} if for each $p\in M^m$,
 there exists a neighborhood $U$ of $p$ and a unit vector field $\nu$
 along $f$ defined on $U$ such that 
 $g\bigl(df(X),\nu\bigr)=0$ holds for any vector field $X$ on $U$
 (that is, $\nu$ is a unit normal vector field),
 and the map $\nu\colon{}U\to T_1N^{m+1}$ is a $C^{\infty}$-map,
 where $T_1N^{m+1}$ is the unit tangent bundle of $N^{m+1}$.
 Moreover, if $\nu$ can be taken to be an immersion for each $p\in M^m$,
 $f$ is called a {\em front} or a {\em wave front}.
 We remark that $f$ is a front if and only if
 $f$ has a lift
 $L_f:M^m\longrightarrow P\bigl(T^*N^{m+1}\bigr)$
 as a Legendrian immersion,
 where $P(T^*N^{m+1})$ is a projectified cotangent bundle on $N^{m+1}$
 with the canonical contact structure.
 The subbundle $\E_f$ which consists of the vectors in
 the pull-back bundle $f^*TN^{m+1}$  
 perpendicular to $\nu$ gives a coherent tangent bundle.
 In fact, 
 $\phi_f\colon{}TM^m\ni X \mapsto df(X)\in \E_f$
 gives a bundle homomorphism.
 Let $\nabla$ be the Levi-Civita connection on $N^{m+1}$.
 Then by taking the tangential part of $\nabla$, it induces 
 a connection $D$ on $\E_f$ satisfying 
 \eqref{eq:c}.
 Let $\inner{~}{~}$ be a metric on $\E_f$ induced from 
 the Riemannian metric on $N^{m+1}$. Then  $D$ is a metric connection
 on $\E_f$. 
 Thus we get a coherent tangent bundle 
 $(M^{m},\E_f,\inner{~}{~},D,\phi_f)$.
 Since the unit tangent bundle can be canonically identified with
 the unit cotangent bundle, 
 the map $\nu\colon{}U\to T_1N^{m+1}$ can be considered
 as a lift of $L_f|_U$.
 A frontal $f$ is called {\it co-orientable\/} if there is a
 unit normal vector field $\nu$ globally defined on $M^m$.
 When $N^{m+1}$ is orientable, the coherent tangent bundle is 
 co-orientable if and only if so is $f$.
\end{example}

From now on, we assume that 
$(M^m,\E,\inner{~}{~},D,\phi)$ is co-orientable, and
fix a co-orientation $\mu$ on the coherent tangent bundle.
(If $\E$ is not co-orientable, one can take a double cover 
 $\pi\colon{}\widehat M^m\to M^m$ such that the pull-back of $\E$ by $\pi$ is
 a co-orientable coherent tangent bundle over $\widehat M^m$.)
\begin{definition}[\cite{SUY_JGEA}]\label{def:volume}
 The {\em signed $\phi$-volume form\/} $d\hat A_{\phi}$ and the 
 ({\em unsigned}) {\em $\phi$-volume form \/} $dA_{\phi}$  
 are defined  as
 \begin{equation}\label{eq:volume-form}
 \begin{aligned}
    d\hat A_\phi := \phi^*\mu = 
      \lambda_{\phi}\,du_1\wedge \dots \wedge du_m,\quad
    dA_\phi 
     := |\lambda_{\phi}|\,du_1 \wedge \dots \wedge du_m,
 \end{aligned}
 \end{equation}
 where $(U;u_1,\dots,u_m)$ is a local coordinate system of $M^m$
 compatible with the orientation of $M^m$, and 
 \begin{equation}\label{eq:jacobian}
    \lambda_{\phi}=\mu\left(
                \phi_{1},\dots ,
                \phi_{m}
      \right)
    \qquad
    \left(\phi_{j}=
         \phi\left(
         \frac{\partial}{\partial u_j}\right),
	 ~
         j=1,\dots,m\right).
 \end{equation}
 We call the function $\lambda_{\phi}$ 
 the {\em $\phi$-Jacobian function\/}  on $U$.
 The set of $\phi$-singular points on $U$ is expressed as
 \begin{equation}\label{eq:singular}
   \Sigma_{\phi}\cap U:=\{p\in U\,;\,\lambda_{\phi}(p)=0  \}. 
 \end{equation}
 Both $d\hat A_{\phi}$  and $dA_{\phi}$ are independent of the  choice of
 positively  oriented local coordinate system $(U;u_1,\dots,u_m)$, 
 and give two globally defined $m$-forms on $M^m$.
 ($d\hat A_{\phi}$ is $C^\infty$-differentiable, 
 but $dA_{\phi}$ is only
 continuous.)
 When $M^m$ has no $\phi$-singular points, the two forms 
 coincide up to sign. We set
 \begin{align*}
   M^+_{\phi}&:=\bigl\{p\in M^m\setminus \Sigma_{\phi} \,;\, 
           d\hat A_{\phi}(p)=dA_{\phi}(p)\bigr\}, \\
   M^-_{\phi}&:=\bigl\{p\in M^m\setminus \Sigma_{\phi}\,;\, 
           d\hat A_\phi(p)=-dA_{\phi}(p)
 \bigr\}.
 \end{align*}
 The $\phi$-singular set $\Sigma_{\phi}$ coincides with
 the boundary  $\partial M^+_\phi=\partial M^-_\phi$.
\end{definition}

A $\phi$-singular point $p$ $(\in\Sigma_{\phi})$ is called 
{\em non-degenerate\/} if $d\lambda_\phi$ does not vanish at $p$.
On a neighborhood of a non-degenerate $\phi$-singular point,
the $\phi$-singular set consists of an $(m-1)$-submanifold in $M^m$, 
called the {\em $\phi$-singular submanifold}.
If $p$ is a non-degenerate $\phi$-singular point, the rank of $\phi_p$ is $m-1$.
The direction of the kernel of $\phi_p$ is called the {\em  null direction}.
Let $\eta$ be the smooth (non-vanishing) vector field along the
$\phi$-singular submanifold $\Sigma_\phi$, which gives the null direction
at each point in $\Sigma_\phi$.

\begin{definition}[$A_2$-singular points, \cite{SUY_JGEA}]%
\label{def:a2-point}
 Let $(M^m,\E,\inner{~}{~},D,\phi)$ be a coherent tangent bundle.
 A non-degenerate $\phi$-singular point $p\in M^m$
 is called an {\em $A_2$-singular point\/} or
 an {\em $A_2$-point of $\phi$\/} 
 if the null direction $\eta(p)$ is transversal to the singular
 submanifold.
\end{definition}

We set
\begin{equation}\label{eq:lambda-prime}
       \lambda_\phi':=d\lambda_\phi(\tilde\eta),
\end{equation}
where $\tilde\eta$ is a vector field on a neighborhood $U$ of $p$
which coincides with $\eta$ on $\Sigma_\phi\cap U$.
Then $p$ is an $A_2$-point if and only if the function
$\lambda'_\phi$ does not vanish at $p$ (see \cite[Theorem 2.4]{SUY4}).

When $m=2$ and $(M^2,\E,\inner{~}{~},D,\phi)$ comes from a
front in $3$-manifold as in Example~\ref{ex:front}
(resp.\ a map into $2$-manifold as in Example~\ref{ex:map}), 
an $A_2$-point corresponds to a cuspidal edge (resp.\ a fold)
(cf.\ \cite{SUY3}).

\subsection{Singular curvatures}
Let $(M^m,\E,\inner{~}{~},D,\phi)$ be a coherent tangent bundle
and fix  a $\phi$-singular point $p\in\Sigma_{\phi}$ which is an
$A_2$-point.
Then there exists a neighborhood $U$ of $p$ such that 
$\Sigma_{\phi}\cap U$ consists of $A_2$-points.
Now we define the {\it singular shape operator\/} as follows:
Since the kernel of $\phi_p$ is transversal to $\Sigma_\phi$ at
$p$, $\phi|_{T(\Sigma_{\phi}\cap U)}$ is injective,
where $U$ is a sufficiently small neighborhood of $p$.
Then the metric $ds^2_\phi$ is positive definite on $\Sigma_\phi\cap U$.
We take an orthonormal frame field
$e_1$, $e_2$,\dots, $e_{m-1}$
on $\Sigma_{\phi}\cap U$ with respect to $ds^2_\phi$.
Without loss of generality, we may assume that 
$(e_1,e_2,\dots,e_{m-1})$ is
smoothly extended on $U$ as an orthonormal $(m-1)$-frame field.
Then we can take a unique smooth section
$\vect{n}:U\to \E$ (called the {\it conormal vector field})
so that
$(\phi(e_1),\dots,\phi(e_{m-1}), \vect{n})$
gives a positively oriented orthonormal frame field on $\E$.
Now, we set
\begin{equation}\label{eq:sing-shape}
   S_\phi(X):=
     -\sgn\left(d\lambda_\phi\bigl(\eta(q)\bigr)\right) 
         \phi^{-1}(D_X\vect n)
     \quad (X\in T_q\Sigma_\phi,\,\,q\in \Sigma_\phi\cap U),
\end{equation}
where the non-vanishing null vector field $\eta$ is chosen so that
$(e_1,\dots,e_{m-1},\eta)$ is compatible with respect to
the orientation of $M^m$.
It holds that 
\begin{equation}\label{eq:sign2}
   \sgn\bigl(d\lambda_{\varphi}(\eta(q))\bigr)
   =
   \begin{cases}
     \hphantom{-} 1 & 
        \mbox{if $\eta(q)$ points toward $M^+_{\phi}$},\\
     -1 & 
        \mbox{if $\eta(q)$ points toward $M^-_{\phi}$}.
   \end{cases}
\end{equation}
Since $\phi$ is injective on each tangent space of
$\Sigma_\phi$ and $D_X\vect{n}\in \phi(T\Sigma_\phi)$,
the inverse element $\phi^{-1}(D_X\vect{n})$ is uniquely determined. 
Thus we get a bundle endomorphism
$S_\phi:T\Sigma_\phi \to T\Sigma_\phi$
which is called the {\it singular shape operator\/} on $\Sigma_\phi$.

\begin{fact}[\cite{SUY_JGEA}]
\label{prop:sym}
 The definition of the singular shape operator $S_\phi$ is independent
 of the choice of an  orthonormal frame field 
 $e_1,\dots,e_{m-1}$,
 the choice of an  orientation of $M^m$,
 and the choice of a co-orientation of $\E$.
 Moreover, it holds that
 \[
   ds^2_\phi\bigl(S_\phi(X),Y\bigr)=ds^2_\phi\bigl(X,S_\phi(Y)\bigr)
        \qquad (X,Y\in T_q\Sigma_\phi,~q\in \Sigma_\phi),
 \]
 namely, $S_\phi$ is  symmetric with respect to $ds^2_{\varphi}$.
\end{fact}

\begin{definition}[\cite{SUY_JGEA}]\label{def:singular-curvature}
 Let $p\in\Sigma_\phi$ be an $A_2$-point of $\phi$.
 Then
 \begin{equation}\label{eq:kappa}
  \kappa_\phi(X):=ds^2_\phi(S_\phi(X),X)/ds^2_\phi(X,X),
   \qquad (X\in T_p\Sigma_{\phi}\setminus\{0\})
 \end{equation}
 is called the {\it $\phi$-singular normal curvature\/}
 at $p$ with respect to the direction $X$.
 The eigenvalues of $S_\phi$ are called the 
 {\it $\phi$-singular principal curvatures\/}, 
 which give the critical values of the singular normal curvature on
 $T_p\Sigma_\phi$.
\end{definition}

In \cite[Theorem 2.13]{SUY_JGEA},
it was shown that at least one of the
$\phi$-singular principal curvatures
diverges to $-\infty$ at non-degenerate singular points
other than $A_2$-points.
When $m=2$, the $\phi$-singular principal curvature 
is called (simply) the {\it $\phi$-singular curvature},
which is also denoted by $\kappa_\phi$.
This definition of the singular curvature
is the same as in \cite[(1.7)]{SUY1} and \cite[(1.6)]{SUY2}.
More precisely, $\kappa_\phi$ is computed as follows:
Let $p\in\Sigma_\phi$ be an $A_2$-point of $\phi$.
Then the $\phi$-singular set $\Sigma_\phi$ is parametrized 
by a regular curve $\gamma(t)$ ($t\in I\subset\R$)
on $M^2$ on a neighborhood of $p$, 
and $\gamma(t)$ is an $A_2$-point of $\phi$ for each $t\in I$.
Since $\dot\gamma(t)$ ($\dot{~}=d/dt$) is not a null-direction,
$\phi\bigl(\dot\gamma(t)\bigr)\neq 0$.
Take a section $\vect{n}(t)$ of $\E$ along $\gamma$ such that
$\{\phi(\dot\gamma)/|\phi(\dot\gamma)|,\vect{n}\}$
gives a positive orthonormal frame field on $\E$ along $\gamma$,
where $|\phi(\dot\gamma)|=\inner{\phi(\dot\gamma)}{\phi(\dot\gamma)}^{1/2}$.
Then we have
\begin{equation}\label{eq:singular-curvature-2}
  \kappa_\phi(t) 
   :=\kappa_{\phi}\bigl(\dot\gamma(t)\bigr)
     = -\sgn \left(d\lambda_\phi\bigl(\eta(t)\bigr)\right)
    \frac{\inner{D_{d/dt}\vect{n}(t)}{\phi\bigl(\dot\gamma(t)\bigr)}}{%
            |\phi\bigl(\dot\gamma(t)\bigr)|^2},
\end{equation}
where $\eta(t)$ is a null-vector field along $\gamma(t)$ such that
$\{\dot\gamma(t),\eta(t)\}$ is compatible with the orientation of $M^2$.
By \eqref{eq:sign2}, 
it holds that 
\begin{equation}\label{eq:sign3}
   \sgn\bigl(d\lambda(\eta(t))\bigr)
   =\begin{cases}
    \hphantom{-}1 
       & \mbox{if  $M^+_{\phi}$ lies on 
               the left-hand side of $\gamma$},\\
     -1 
       & \mbox{if  $M^-_{\phi}$ lies on 
               the left-hand side of $\gamma$}.
\end{cases}
\end{equation}

\section{The realization of frontal bundles}
\label{sec:realization}

First, we recall a definition of frontal bundles 
given in \cite{SUY_JGEA},
and 
consider a realization problem of them as  fronts in space forms.

\subsection{Front bundles}
Let $M^m$ be an  oriented $m$-manifold and 
$(M^m,\E,\inner{~}{~},D,\phi)$ 
a co-orientable coherent tangent bundle over $M^m$.
If there exists another  bundle homomorphism $\psi:TM^m\to \E$
such  that $(M^m,\E,\inner{~}{~},D,\psi)$
is also a coherent tangent bundle and
the pair $(\phi,\psi)$ of bundle homomorphisms satisfies a compatibility
condition
\begin{equation}\label{eq:compati}
  \inner{\phi(X)}{\psi(Y)}=\inner{\phi(Y)}{\psi(X)},
\end{equation}
then $(M^m,\E,\inner{~}{~},D,\phi,\psi)$ is called a {\it frontal bundle}.
The bundle homomorphisms $\phi$ and $\psi$ are called
the {\em first homomorphism\/} and the {\em second homomorphism},
respectively.
We set
\begin{align*}
 \first(X,Y)&:=ds^2_{\phi}(X,Y)=\inner{\phi(X)}{\phi(Y)}, \\
 \second(X,Y)&:=-\inner{\phi(X)}{\psi(Y)},\\
 \third(X,Y)&:=ds^2_{\psi}(X,Y)= \inner{\psi(X)}{\psi(Y)} 
\end{align*}
for $X,Y\in T_pM^m$ ($p\in M^m$),  and we call them 
{\it the first, the second and the third fundamental forms},
respectively. 
They are all symmetric covariant tensors on $M^m$.

\begin{definition}[\cite{SUY_JGEA}]\label{def:front}
 A frontal bundle $(M^m,\E,\inner{~}{~},D, \phi,\psi)$
 is called  a {\it front bundle\/} if
 \begin{equation}\label{eq:front}
  \Ker(\phi_p)\cap \Ker(\psi_p)=\{0\}
 \end{equation}
 holds for each   $p\in M^m$.
\end{definition}

\begin{example}[\cite{SUY_JGEA}]\label{ex:front-bundle}
 Let $\bigl(N^{m+1}(c),g\bigr)$ be 
 an $(m+1)$-dimensional space form,
 that is, a complete Riemannian $(m+1)$-manifold of constant curvature
 $c$,
 and denote by $\nabla$ the Levi-Civita connection 
 on $N^{m+1}(c)$.
 Let $f:M^m\to N^{m+1}(c)$ be a co-orientable frontal.
 Then there exists a globally defined unit normal vector field $\nu$.
 Since the coherent tangent bundle $\E_f$ given in
 Example~\ref{ex:front} is orthogonal to $\nu$,
 we can define a bundle homomorphism
 \[
    \psi_f:T_pM^m\ni X\longmapsto \nabla_X\nu \in \E_p
            \qquad (p\in M^m).
 \]
 Then $(M^m,\E_f,\inner{~}{~},D,\phi_f,\psi_f)$ is a frontal bundle
 (we shall prove this in Proposition~\ref{prop:frontal} later).
 Moreover, this is a front bundle in the sense of
 Definition~\ref{def:front} if and only if $f$ is a front,
 which is equivalent to $\first+\third$ being positive definite.
\end{example}

\begin{remark}
\label{rem:meaning}
 As seen above, if $f:M^m\to N^{m+1}(c)$ is a front,
 then 
 \[
    (M^m,\E_f,\inner{~}{~},D,\phi_f,\psi_f)
 \]
 is a front bundle.
 Since $\phi=\phi_f$ and $\psi=\psi_f$ have the completely same
 conditions,
 the third fundamental form $\third$ plays the same role as
 $\first$ by definition.
 This means that we can reverse the role of $\first$ and $\third$.

 When $N^{m+1}(c)$ is the unit sphere $S^{m+1}$, then
 the unit normal vector field $\nu$ along $f$ can be considered
 as a map $\nu:M^m\to S^{m+1}$ and the third fundamental form of $f$
 coincides with 
 the first fundamental form of $\nu$.

 When $N^{m+1}(c)$ is the Euclidean space $\R^{m+1}$, then
 the unit normal vector field $\nu$ along $f$ can be considered
 as a map $\nu:M^m\to S^{m}$ and the
 third fundamental form of $f$ coincides with 
 the pull-back of the canonical metric of the unit sphere $S^m$ by $\nu$.

 Next, we cosider the case that $N^{m+1}(c)$ is the hyperbolic space
 $H^m$:
 \begin{equation}\label{eq:hyperbolic}
    H^{m+1}:=\{p=(p_0,\dots,p_{m+1})
        \in \R^{m+2}_1\,;\,p\cdot p=-1,~p_0>0\},
 \end{equation}
 where \lq$\cdot$\rq\ is the canonical Lorentzian
 metric of the Lorentz-Minkowski space $\R^{m+2}_1$. 
 The unit normal vector field $\nu$ along $f$ can be considered
 as a map $\nu:M^m\to S^{m+1}_1$ and the
 third fundamental form of $f$ coincides with
 the first fundamental form of $\nu$, where 
 \begin{equation}\label{eq:de-sitter}
    S^{m+1}_1:=\{p\in \R^{m+2}_1\,;\,p\cdot p=1\}
 \end{equation}
 is the  de Sitter space form.
\end{remark}

\begin{proposition}
\label{prop:frontal}
 Let $f: M^m\to N^{m+1}(c)$ be a  co-orientable  frontal,
 and $\nu$ a unit normal vector field.
 Then 
 $(M^m,\E_f,\inner{~}{~}, D, \phi_f,\psi_f)$
 as in Example~\ref{ex:front-bundle} is a frontal bundle.
 Moreover, the following identity 
 {\rm (}i.e. the Gauss equation{\rm)}
 holds{\rm:}
 \begin{multline}\label{eq:G}
  \langle R^D(X,Y)\xi,\zeta\rangle\\=
  c
  \det\pmt{%
    \inner{\phi(Y)}{\xi}& \inner{\phi(Y)}{\zeta} \\
    \inner{\phi(X)}{\xi}& \inner{\phi(X)}{\zeta} 
  }
  +
  \det\pmt{%
    \inner{\psi(Y)}{\xi}& \inner{\psi(Y)}{\zeta} \\
    \inner{\psi(X)}{\xi}& \inner{\psi(X)}{\zeta} 
  },
 \end{multline}
 where $\phi=\phi_f$ and $\psi=\psi_f$, 
 $X$ and $Y$ are vector fields on $M^m$,
 $\xi$ and $\zeta$ are sections of $\E_f$,
 and $R^D$ is the curvature tensor of the connection $D${\rm:}
 \[
    R^D(X,Y)\xi:=D_XD_Y\xi-D_YD_X\xi-D_{[X,Y]}\xi.
 \]
 Furthermore, this frontal bundle 
 is a front bundle if and only if $f$ is a front.
\end{proposition}
\begin{proof}
 Let $R^c$ be the curvature tensor of $N^{m+1}(c)$.
 Since 
 \[
    \nabla_X\xi=D_X\xi-\langle \psi_f(X),\xi \rangle \nu
 \]
 holds for the Levi-Civita connection $\nabla$
 of $N^{m+1}(c)$,
 we have the following identity:
 \begin{multline}\label{eq:curvature-tensor}
  R^c\bigl(df(X),df(Y)\bigr)\xi 
  =R^D(X,Y)\xi
    -\inner{\psi_f(Y)}{\xi}\psi_f(X) 
    +\inner{\psi_f(X)}{\xi}\psi_f(Y) \\
     -\bigl(\inner{D_X\psi_f(Y)}{\xi} -
     \inner{D_Y\psi_f(X)}{\xi}
     -\inner{\psi_f([X,Y])}{\xi}\bigr) \nu.
 \end{multline}
 Taking the normal component, we get 
 \[
  \inner{D_X\psi_f(Y)}{\xi}-
  \inner{D_Y\psi_f(X)}{\xi}
   =\inner{\psi_f\bigl([X,Y]\bigr)}{\xi}.
 \]
 Since $\xi$ is  arbitrary, this  proves that
 $(M^m,\E_f,\inner{~}{~},D,\psi_f)$ is a coherent tangent bundle.
 Moreover,
 \[
    \inner{\phi_f(X)}{\psi_f(Y)}=
        g\bigl(
        df(X),\nabla_X\nu
        \bigr)
        =\inner{\phi_f(Y)}{\psi_f(X)}.
 \]
 Hence $(M^m,\E_f,\inner{~}{~},D,\phi_f,\psi_f)$ is a frontal bundle.

 On the other hand,
 taking the tangential component of \eqref{eq:curvature-tensor},
 we get 
 \[
  R^c\bigl(df(X),df(Y)\bigr)\xi=
  R^D(X,Y)\xi
    -\inner{\psi_f(Y)}{\xi}\psi_f(X) 
     +\inner{\psi_f(X)}{\xi} \psi_f(Y). 
 \]
 Since $(N^{m+1}(c),g)$ is of constant curvature $c$,
 it holds that
 \[ 
  R^c\bigl(df(X),df(Y)\bigr)\xi=
   c\biggl(
       \inner{\phi_f(Y)}{\xi}\phi_f(X) 
        -\inner{\phi_f(X)}{\xi} \phi_f(Y)
     \biggr ),
 \]
 and hence we get the Gauss equation \eqref{eq:G}.
\end{proof}

\begin{definition}\label{def:integrability}
 For a real number $c$,
 a frontal bundle $(M^m,\E,\inner{~}{~},D,\phi,\psi)$ is said to be 
 {\em $c$-integrable\/} if and only if
 \eqref{eq:G} holds.
\end{definition}

\subsection{A realization of frontal bundles}
Now, we give the fundamental theorem for frontal bundles.
To state the theorem, we define equivalence of frontal bundles:
\begin{definition}\label{def:isomorphic}
 Two frontal bundles 
 over $M^m$  are {\em isomorphic\/} or {\em equivalent\/} if 
 there exists an orientation preserving
 bundle isomorphism between them
 which preserves the inner products, the connections and the bundle
 maps.
\end{definition}
Let $\bigl(\widetilde N^{m+1}(c),g\bigr)$ be the $(m+1)$-dimensional
{\em simply connected\/} space form of constant curvature $c$.
\begin{theorem}[Realization of frontal bundles]%
\label{thm:fundamental}
 Let  $(U,\E,\inner{~}{~},D,\phi,\psi)$ be a $c$-in\-te\-grable
 frontal bundle over a simply connected domain $U\subset\R^m$,
 where $c$ is a real number.
 Then there exists a frontal
 $f\colon{}U \to \widetilde{N}^{m+1}(c)$
 such that $\E$ is isomorphic to 
 $\E_{f}$  induced from $f$ as in Proposition~\ref{prop:frontal}.
 Moreover, such an $f$ is unique up to orientation preserving isometries
 of $\widetilde N^{m+1}(c)$.
\end{theorem}
Let $S^{m+1}_1$ be the de Sitter space of constant
sectional curvature $1$. 
As mentioned in Remark \ref{rem:meaning},
$S^{m+1}_1$ can be identified with the hyperquadric in 
the Lorentz-Minkowski space $\R^{m+2}_1$ 
(see \eqref{eq:de-sitter}).
A $C^\infty$-map $f:M^m\to S^{m+1}_1$ is called a {\it frontal\/}
if there exists a  $C^\infty$-map 
\[
   \nu:M^m\longrightarrow H^{m+1}:=\{p=(p_0,\dots,p_{m+1})
       \in \R^{m+2}_1\,;\,p\cdot p=-1,\,\,p_0>0\}
\]
such that
$d\nu\cdot f=\nu\cdot df=0$.
Moreover, $f$ is called a {\it {\rm(}wave{\rm)} front} if
$(f,\nu):M^m\to \R^{m+2}_1\times \R^{m+2}_1$
is an immersion.
By definition, $f$ is a front if and only if  $\nu$
also is. 
Thus, by interchanging the role of 
the first homomorphism and the second homomorphism,
we get the following
\begin{corollary}
\label{cor:fundamental}
 Let  $(U,\E,\inner{~}{~},D,\psi,\phi)$ be a $(-1)$-integrable
 frontal bundle over a simply connected domain $U\subset\R^m$.
 Then there exists a frontal
 $\nu\colon{}U \to S^{m+1}_1$
 such that $\E$ is isomorphic to  $\E_{f}$  induced from $\nu$.
 Moreover, such an $\nu$ is unique up to orientation 
 preserving isometries  of $S^{m+1}_1$.
\end{corollary}

\begin{proof}[Proof of Theorem~\ref{thm:fundamental}]
 To prove Theorem \ref{thm:fundamental},
 we write down the fundamental equations for frontals.
 First, we consider the case $c=0$.
 Let $f\colon{}U\to \R^{m+1}=\widetilde N^{m+1}(0)$ be a frontal,
 where we consider elements in the Euclidean space $\R^{m+1}$ as 
 column vectors.
 Then the unit normal vector field $\nu$ can be considered as 
 a map $\nu\colon{}U\to S^m\subset\R^{m+1}$, and $\nabla\nu=d\nu$,
 where $\nabla$ is the Levi-Civita connection of $\R^{m+1}$.
 Thus the corresponding frontal bundle is 
 $(U,\E_{f},\inner{~}{~},D,\phi:=df,\psi:=d\nu)$.
 Take a positively oriented orthonormal frame field 
 (called an {\em adopted frame field\/} of $f$)
 \begin{equation}\label{eq:adopted-frame-euc}
    \F:=(\vect{e}_1,\dots,\vect{e}_{m+1}):
        U\longrightarrow \SO(m+1)
 \end{equation}
 of $\R^{m+1}$ along $f$ such that $\vect{e}_{m+1}=\nu$.
 Since $\nu=\vect{e}_{m+1}$, $\{\vect{e}_1,\dots,\vect{e}_m\}$
 is an orthonormal frame field of $\E_{f}$.
 Let $\omega_i^j$ be the {\em connection forms\/} of $D$
 with respect to this basis, as $1$-forms on $U$:
 \begin{equation}\label{eq:conn-form-2}
     D\vect{e}_i = \sum_{l=1}^{m}\omega_i^l\vect{e}_l,\qquad 
      \omega_i^j=-\omega_j^i \quad     (i,j=1,\dots,m).
 \end{equation}
 Define an $\so(m)$-valued $1$-form $\Omega$ 
 by $\Omega=(\omega_i^j)$.
 Next, we define $\R^{m}$-valued $1$-forms
 $\vect{g}$ and $\vect{h}$ as 
 \begin{equation}\label{eq:g-h}
   \vect{g}:=\trans{(g^1,\ldots,g^m)},\qquad 
   \vect{h}:=\trans{(h^1,\ldots,h^m)}
 \end{equation}
 with
 \[
    g^j:=\inner{\phi}{\vect{e}_j},\quad
    h^j:=-\inner{\psi}{\vect{e}_j}\qquad (j=1,\dots,m),
 \]
 where $\R^{m}$ is considered as a column vector space.
 Then, by definition, 
 the adapted frame $\F$ in \eqref{eq:adopted-frame-euc}
 satisfies the ordinary differential equation
 \begin{equation}
   \label{eq:gauss-wein-euc}
     df  = \sum_{l=1}^m g^l\vect{e}_l,\quad 
     d\F = \F \widetilde\Omega,\quad 
   \widetilde\Omega =
    \begin{pmatrix}
       \Omega & -\vect{h} \\
       \trans{\vect{h}} & \hphantom{-}0 
    \end{pmatrix}.
 \end{equation}

 Next, we consider the case $c>0$.
 Without loss of generality, we may assume that $c=1$.
 In this case,
 $\widetilde N^{m+1}(c)$ can be considered as the
 unit sphere  $S^{m+1}(\subset \R^{m+2})$ centered at the origin.
 Let $f\colon{}U\to S^{m+1}$  
 be a frontal with the unit normal vector field
 $\nu\colon{}U\to S^{m+1}$.
 Then the coherent tangent bundle $\E_{f}$ is written as
 \begin{equation}\label{eq:ef}
  \E_{f} =\{\vect{x}\in\R^{m+2}\,;\,\vect{x}\cdot{f}=\vect{x}\cdot\nu=0\},
 \end{equation}
 where 
 ``$\cdot$'' is the canonical inner product of $\R^{m+2}$.
 The induced inner product $\inner{~}{~}$ of $\E_{f}$ is the restriction
 of ``$\cdot$''.
 Take an $\SO(m+2)$-valued function  ({\em an adopted frame})
 $\F:=(\vect{e}_0,\dots,\vect{e}_{m+1})\colon{}U \to \SO(m+2)$
 such that
 $\vect{e}_0 := f$, $\vect{e}_{m+1}:=\nu$.
 Since $d\nu\cdot f=d\nu\cdot\nu=0$,
 $d\nu$ is a  $\E_{f}$-valued $1$-form, and then it holds that
 \[
          \nabla \nu=d\nu,
 \]
 where $\nabla$ is the Levi-Civita connection of $S^{m+1}$.
 Thus, setting $\phi=df$ and $\psi=d\nu$, we have the frontal
 bundle.
 Denoting by $\omega_i^j$ ($i,j=1,\dots,m$) the connection forms 
 of $D$ with respect to $\{\vect{e}_j\}$,
 the adapted frame field $\F$ satisfies 
 \begin{equation}\label{eq:gauss-wein-sphere}
    d\F = \F \widetilde\Omega,\qquad
   \widetilde\Omega =
    \begin{pmatrix}
       0 & -\trans{\vect{g}} & \hphantom{-}0 \\
      {\vect{g}}&\hphantom{-}\Omega & -\vect{h} \\
      0 & \hphantom{-}\trans{\vect{h}} & \hphantom{-}0
    \end{pmatrix},
 \end{equation}
 where $\Omega=(\omega_i^j)$, and $\vect{g}$ and $\vect{h}$ are as 
 in \eqref{eq:g-h} in the case of $c=0$.

 Finally, we consider the case $c<0$.
 We may assume that $c=-1$. Then
 $\widetilde N^{m+1}(-1)$ is the hyperbolic space $H^{m+1}$ 
 as in \eqref{eq:hyperbolic}.
 Let  $f\colon{}M^m\to H^{m+1}$ be a 
 frontal and $\nu$ be the unit normal vector field.
 Then $\nu$ is a space-like frontal in de Sitter space $S^{m+1}_1$
 as in \eqref{eq:de-sitter},
 and the coherent tangent bundle is written 
 like as \eqref{eq:ef}, using the canonical Lorentzian inner product.
 Take an $\SO_0(1,m+1)$-valued function  ({\em an adapted frame})
 $\F:=(\vect{e}_0,\dots,\vect{e}_{m+1})\colon{}U\to \SO_0(1,m+1)$
 such that
 $\vect{e}_0 := f$, $\vect{e}_{m+1}:=\nu$,
 where $\SO_0(1,m+1)$ is the identity component of the group of linear
 isometries $\O(1,m+1)$ of $\R^{m+2}_1$.
 Similar to the case of $c>0$, it holds that $\nabla \nu=d\nu$, and then 
 we can set $\phi=df$, $\psi=d\nu$.
 Hence the adapted frame field $\F$ satisfies 
 \begin{equation}\label{eq:gauss-wein-hyp}
    d\F = \F \widetilde\Omega,\qquad
   \widetilde\Omega =
    \begin{pmatrix}
       0 & \hphantom{-}\trans{\vect{g}} & \hphantom{-}0 \\
      {\vect{g}}&\hphantom{-}\Omega & -\vect{h} \\
      0 & \hphantom{-}\trans{\vect{h}} & \hphantom{-}0
    \end{pmatrix},
 \end{equation}
 as well as the case of $c>1$, where
 $\Omega=(\omega_i^j)$ and $\vect{g}$ and $\vect{h}$ are as in 
 \eqref{eq:g-h}.

 Now, in these situation, 
 the Gauss equation \eqref{eq:G} and 
 the Codazzi equation \eqref{eq:c} for $\psi$
 can be considered as the integrability conditions for
 the differential equations \eqref{eq:gauss-wein-euc} 
 and \eqref{eq:gauss-wein-sphere}.
 Thus we get the assertion.
\end{proof}

We give here several applications of the realization theorem.
\begin{theorem}[%
  Maps into $\widetilde N^m(c)$ of an $m$-dimensional domain]
\label{thm:morin}
 Let $U$  be a simply connected domain on $\R^m$
 and $(U,\E,\inner{~}{~},D,\phi)$  a 
 coherent tangent bundle over $U$.
 Assume that for any vector fields $X$, $Y$ on $U$ and 
 a section  $\xi$ of $\E$, it holds that
 \begin{equation}\label{eq:morin-integrable}
    R^D(X,Y)\xi
    =c\biggl(%
         \inner{\phi(Y)}{\xi}\phi(X)-
	 \inner{\phi(X)}{\xi}\phi(Y)
      \biggr),
 \end{equation}
 where $R^D$ is the curvature tensor of $D$.
 Then there exists a $C^{\infty}$-map 
 $f\colon{}U\to \widetilde N^m(c)$
 into the $m$-dimensional simply connected space form 
 $\widetilde N^m(c)$ such that $\E$ and $\E_{f}$ 
 {\rm (}as in Example ~\ref{ex:map}{\rm)}
 are isomorphic.
\end{theorem}
\begin{proof}
 Consider the trivial bundle map
 $\vect{0}\colon{}TM^{m}\ni X\mapsto \vect{0}\in\E$.
 Then by \eqref{eq:morin-integrable}, 
 $(U,\E,\inner{~}{~},D,\phi,\vect{0})$ is a $c$-integrable 
 frontal bundle, and then there exists
 the corresponding frontal $\tilde f\colon U\to \widetilde  N^{m+1}(c)$. 
 Since $\psi=\vect{0}$,
 the image of $\tilde f$ lies in a totally geodesic hypersurface of 
 $\widetilde N^{m+1}(c)$.
\end{proof}

\subsection{Applications to surface theory}
Now we introduce applications for surface theory.
To state them, we rewrite the $c$-integrability \eqref{eq:G}
for the $2$-dimensional case.
Let $(M^2,\E,\inner{~}{~},D,\phi,\psi)$ be a frontal bundle
over a $2$-manifold $M^2$.
Take a (local) orthonormal frame field $\{\vect{e}_1,\vect{e}_2\}$
of $\E$, and take a $1$-form $\omega$ as 
\begin{equation}\label{eq:conn-form}
    D\vect{e}_1 = -\omega \vect{e}_2,\qquad
    D\vect{e}_2 = \omega\vect{e}_1,
\end{equation}
that is, $\omega$ is the connection form of $D$ with respect to 
the frame $\{\vect{e}_1,\vect{e}_2\}$.
Then one can easily see that
$(M^2,\E,\inner{~}{~},D,\phi,\psi)$ is $c$-integrable if and only if
\begin{equation}\label{eq:2-integrable}
 d\omega = c\alpha + \beta
\end{equation}
holds, where $\alpha$ and $\beta$ are $2$-forms on $M^2$ defined by
\begin{align*}
   \alpha(X,Y) &= 
   \inner{\phi(X)}{\vect{e}_1}
   \inner{\phi(Y)}{\vect{e}_2}-
   \inner{\phi(X)}{\vect{e}_2}
   \inner{\phi(Y)}{\vect{e}_1},\\
   \beta(X,Y) &= 
   \inner{\psi(X)}{\vect{e}_1}
   \inner{\psi(Y)}{\vect{e}_2}-
   \inner{\psi(X)}{\vect{e}_2}
   \inner{\psi(Y)}{\vect{e}_1}.
\end{align*}
\begin{remark}
\label{rem:d-omega}
 Let $K_\phi$ be the Gaussian curvature of the first fundamental
 form $\first=ds_\phi^2$.
 Then 
 \begin{equation}\label{eq:d-omega}
     d\omega = K_\phi\,d\hat A_\phi
 \end{equation}
 holds, where $d\hat A_\phi$ is the signed $\phi$-volume form
 defined in Definition~\ref{def:volume}.
\end{remark}

\begin{theorem}[Fronts of constant negative extrinsic curvature]
\label{thm:sine-gordon}
 Let $U$ be a simply connected domain of $\R^2$
 and $c\in\R$  a constant.
 Take a smooth real-valued function $\theta=\theta(u,v)$ on $U$
 which satisfies the equation{\rm:}
 \begin{equation}\label{eq:c-sine-gordon}
    \theta_{uv} = (c-1)\sin\theta,
 \end{equation}
 where $\theta_{uv}:=\partial^2\theta/(\partial u\partial v)$.
 Then there exists a front
 $f\colon{} U \to \widetilde N^3(c)$
 whose fundamental forms are given by
 \begin{equation}\label{eq:f-form-neg}
 \begin{alignedat}{2}
   \first &= \inner{\phi}{\phi}&=& du^2 + 2\cos\theta\,du\,dv +dv^2,\\
   \second &=-\inner{\phi}{\psi}&=& 2 \sin\theta\,du\,dv,\\
   \third &=\inner{\psi}{\psi} &=&du^2 - 2\cos\theta\,du\,dv +dv^2.
 \end{alignedat}
 \end{equation}
 In particular, the Gaussian curvature of $f$ is identically $c-1$ on
 $U\setminus \Sigma$, where $\Sigma=\{\theta\equiv 0 \pmod\pi\}$ 
 is the singular set of $f$.
 Conversely, any front 
 $f\colon{} U \to \widetilde N^3(c)$ whose regular set
 $R_f:=U\setminus\Sigma$ is dense in $U$ and whose  Gaussian curvature
 is $c-1$ on $R_f$ is given in this manner.
\end{theorem}

\begin{proof}
 Let  $\E=U\times\R^2$ be the trivial bundle and take 
 the canonical orthonormal frame
 $\{\vect{a}_1,\vect{a}_2\}$.
 Define the bundle homomorphisms $\phi$ and $\psi$
 from $TU$ to $\E$ as 
 \begin{equation}\label{eq:chebyshef-homo}
  \begin{aligned}
   \phi&:=\cos\frac{\theta}2(du+dv)\vect{a}_1-
   \sin\frac{\theta}2(du-dv)\vect{a}_2, \\
   \psi&:=-\sin\frac{\theta}2(du+dv)\vect{a}_1
          -\cos\frac{\theta}2(du-dv)\vect{a}_2. 
  \end{aligned}
 \end{equation}
 Take a connection $D$ of $\E$ as 
 \begin{equation}\label{eq:chebyshef-conn}
    D\vect{a}_1 = -\omega \vect{a}_2,\quad 
    D\vect{a}_2 = \omega \vect{a}_1,\qquad 
    \omega = \frac{1}{2}\!
         \left(\theta_u du-
               \theta_v dv\right).
 \end{equation}
 Then by \eqref{eq:2-integrable} and \eqref{eq:c-sine-gordon},
 $(U,\E,\inner{~}{~},D,\phi,\psi)$ is a $c$-integrable front bundle,
 and hence we have the corresponding front $f$.
 In particular, the fundamental forms of $f$ are
 given by
 \eqref{eq:f-form-neg}.
 Hence
 $(u,v)$ forms an asymptotic Chebyshev net of $f$, and the Gaussian
 curvature
 is $(c-1)$.
 Moreover, $\theta$ is the angle between two asymptotic directions 
 with respect to the first fundamental form.

 Conversely, suppose that 
 $f\colon{} U \to \widetilde N^3(c)$ is a front such that the regular
 set $R_f$ of $f$ is dense in $U$ and $f$ 
 has constant Gaussian curvature $(c-1)$ on $R_f$.
 Then the sum $\first+\third$ of the first and 
 the third fundamental forms is a flat metric.
 Since $U$ is simply connected, there is an
 immersion
 $\Phi:U\to (\R^2;u,v)$
 such that $\first+\third=\Phi^*\bigl(2(du^2+dv^2)\bigr)$.
 The asymptotic lines of $f$ on $R_f$
 are geodesic lines with respect to the metric $\first+\third$, and
 two asymptotic directions are mutually
 orthogonal
 with respect to the metric $\first + \third$.
 Thus by rotating the coordinate system $(u,v)$,
 we may assume that the inverse image of $u,v$-lines
 by $\Phi$ consists on asymptotic lines.
 Then the fundamental forms are given by 
 \eqref{eq:f-form-neg} on $\Phi(R_f)$. 
 Since $R_f$ is a dense set, 
 \eqref{eq:f-form-neg} holds on $\Phi(U)$, which proves the assertion.
\end{proof}

In particular, we have
the following assertion on the realization 
of fronts of constant negative curvature $-1$
in $\R^3$
and flat front in $S^3$, respectively.
\begin{corollary}
 Let $U$ be a simply connected domain of $\R^2$, and 
 take a smooth real-valued function $\theta$ on $U$
 which satisfies
 \begin{equation}\label{eq:sine-gordon}
  \theta_{uv}=\sin\theta\qquad
   (\text{resp.\ }\theta_{uv}=0).
 \end{equation}
 Then there exists a front
 $f\colon{} U \to \R^3$ {\rm(}resp.\ $S^3${\rm)}
 such that the Gaussian curvature of $f$ is identically $-1$ 
 {\rm(}resp.\ $0${\rm)} on
 $U\setminus \Sigma$, where $\Sigma=\{\theta\equiv 0 \pmod\pi\}$ 
 is the singular set.
\end{corollary}

\begin{theorem}[Fronts of constant positive curvature]
\label{thm:sinh-gordon}
 Let $U$ be a simply connected domain of $\C=\R^2$, and 
 take a smooth real-valued function $\theta$ on $U$
 which satisfies the sinh-Gordon equation{\em:}
 \begin{equation}\label{eq:sinh-gordon}
     \frac{1}{4}
      \left(
    \theta_{uu}+
       \theta_{vv}
      \right)
      (=\theta_{z\bar z})=-\sinh\theta,
 \end{equation}
 where $z=u+i v$ is the complex coordinate on $\C=\R^2$.
 Then there exists a front
 $f\colon{} U \to\R^3$
 without umbilic points, 
 whose fundamental forms are given by
 \begin{equation}\label{eq:f-form-pos}
 \begin{array}{llll}
   \first &= 
    \inner{\phi}{\phi}&=&\,
         dz^2 + 2\cosh\theta\,dz\,d\bar z  +d\bar z^2,\\
   &&=&\,4\left\{
           \cosh^2({\theta}/{2})\,du^2 + \sinh^2({\theta}/{2})\,dv^2
        \right\},\\
   \second &=-\inner{\phi}{\psi}&=&\,4\sinh\theta\,dz\,d\bar z,\\
   &&=&\,4\cosh({\theta}/{2})\sinh({\theta}/{2})\left(du^2+dv^2\right),\\
   \third &= 
    \inner{\psi}{\psi}&=&\, 
         -dz^2 + 2\cosh\theta\,dz\,d\bar z  -d\bar z^2,\\
   &&=&\,4\left\{
           \sinh^2({\theta}/{2})\,du^2 + \cosh^2({\theta}/{2})\,dv^2
        \right\}.
 \end{array}
 \end{equation}
 Conversely, any front 
 $f\colon{} U \to \R^3$ whose regular set
 $R_f=U\setminus\Sigma$ is dense in $U$ and whose  Gaussian curvature
 is $1$ on $R_f$ without umbilic points is given in this manner.
\end{theorem}

\begin{proof}
 Let  $\E=U\times\R^2$ be the trivial bundle and take 
 the canonical orthonormal frame
 $\{\vect{a}_1,\vect{a}_2\}$.
 Define the bundle homomorphisms $\phi$ and $\psi$ as 
 \begin{equation}\label{eq:complex-chebyshef-homo}
  \begin{aligned}
   \phi:&= 2\left[
             \left(\cosh\frac{\theta}{2}du\right)\vect{a}_1 +
             \left(\sinh\frac{\theta}{2}dv\right)\vect{a}_2
            \right] \\
        &=  \cosh\frac{\theta}{2}(dz+d\bar z)\vect{a}_1 - 
            i\sinh\frac{\theta}{2}(dz-d\bar z)\vect{a}_2, \\
   \psi:&= -2\left[ 
           \left(\sinh\frac{\theta}{2}du\right)\vect{a}_1 +
             \left(\cosh\frac{\theta}{2}dv\right)\vect{a}_2
           \right]
           \\
        &=  -\sinh\frac{\theta}{2}(dz+d\bar z)\vect{a}_1  
            +i\cosh\frac{\theta}{2}(dz-d\bar z)\vect{a}_2,
  \end{aligned}
 \end{equation}
 and define  a connection $D$ on $\E$ by a connection form
 \begin{equation}\label{eq:complex-chebyshef-conn}
  \omega= \frac{1}{2}\left(
              \theta_v du-
              \theta_u dv
             \right)
        = \frac{i}{2}\left(
              \theta_z dz-
              \theta_{\bar z}d\bar z
             \right).
 \end{equation}
 Thus by \eqref{eq:2-integrable}, 
 $(U,\E,\inner{~}{~},D,\phi,\psi)$ is a $0$-integrable front bundle,
 and then we have the corresponding front $f$.
 In particular, the fundamental forms of $f$ are
 given by \eqref{eq:f-form-pos}.
 Hence 
 $(u,v)$ forms a curvature line coordinate system,
 and the Gaussian  curvature is $1$.

 Conversely, suppose that $f\colon{} U \to \R^3$
 is a front such that the regular set $R_f$ of $f$ is dense in $U$
 and $f$ has constant Gaussian curvature $1$ on $R_f$.
 Then $\first-\third$ gives a flat Lorentzian metric.
 Since $U$ is simply connected, there is an
 immersion
 $\Phi:U\to (\R^2;u,v)$
 such that $\first-\third=\Phi^*\bigl(4(du^2-dv^2)\bigr)$.
 The curvature lines of $f$ on $R_f$ are geodesic lines with respect to
 the metric $\first-\third$, and the
 two principal directions are orthogonal with respect to
 $\first-\third$. 
 Thus by Lorentzian rotation of the coordinate system $(u,v)$,
 we may assume that the inverse image of $u,v$-lines
 under $\Phi$ consists of principal curvature lines.
 Then the fundamental forms are given by \eqref{eq:f-form-pos} on
 $\Phi(R_f)$. 
 Since $R_f$ is a dense set, 
 \eqref{eq:f-form-pos} holds on $\Phi(U)$, which proves the assertion.
\end{proof}

\section{%
 A relationship between sectional curvatures and 
 singular principal curvatures
}
\label{sec:bdd}
In this section, we investigate a relationship between
sectional curvatures (cf. \eqref{eq:ext2})
near $A_2$-singular points of hypersurfaces (as wave fronts) and 
their singular principal curvatures.

We fix a front bundle $(M^m,\E,\inner{~}{~},D,\phi,\psi)$
over an $m$-dimensional manifold $M^m$.
\begin{definition}\label{def:ext}
 When $p\in M^m$ is not a singular point of $\phi$, we define
 \begin{equation}\label{eq:ext}
      K^{\ext}(X\wedge Y):=
      \frac{
        \second(X,X) \second(Y,Y)
            -\second(X,Y)^2
      }{
         I(X,X) I(Y,Y)
             -I(X,Y)^2
      }
     \qquad (X,Y\in T_pM^m),
 \end{equation}
 which is called the {\em extrinsic curvature\/} at $p$ 
 with respect to the $X\wedge Y$-plane in $T_pM^m$.
\end{definition}

If a front bundle $(M^m,\E,\inner{~}{~},D,\phi,\psi)$
is induced from a front in $N^{m+1}(c)$,
then it holds that
 \begin{equation}\label{eq:ext2}
      K^{\ext}(X\wedge Y):=K(X\wedge Y)+c
     \qquad (X,Y\in T_pM^m),
 \end{equation}
where $K(X\wedge Y)$ is the sectional curvature
at each $\phi$-regular point $p$ of $M^m$. 
Theorem 0.1 given in the introduction is
a direct consequence of the following assertion:

\begin{theorem} \label{thm:bdd}
 Let  $(M^m,\E,\inner{~}{~},D,\phi,\psi)$ be a front bundle over
 an oriented $m$-manifold $M^m$.
 Take an $A_2$-point $p\in M^m$ of $\phi$.
 Then the following hold{\rm :}
\begin{enumerate} 
 \item\label{item:bdd:1}
  Suppose that $K^{\ext}$ is bounded 
  except on the singular set near $p$. 
  Then $\second(X,Y)=0$ holds for all $X$, $Y\in T_pM^m$.
 \item\label{item:bdd:2}
  If $K^{\ext}$ does not change sign on a neighborhood of $p$
  with the singular set removed,       
  then $K^{\ext}$ is bounded on that neighborhood of $p$ with
   the singular set removed.
 \item\label{item:bdd:3}
  If $K^{\ext}$ is non-negative except on the singular set
  near $p$, then the singular principal curvatures 
  at $p$ are all non-positive.
  Furthermore, if there exists a $C^\infty$ vector field $\tilde\eta$
  defined on a neighborhood $U$ of $p$ and a constant $\delta>0$
  such that the restriction of $\tilde \eta$ 
  on $U\cap\Sigma_{\phi}$ gives a null vector field,
  and $K^{\ext}(X\wedge \tilde \eta)\ge \delta$
  holds on $U\setminus \Sigma_\phi$
  for each $C^\infty$-vector field $X$ on $U$
  satisfying $X\wedge \tilde \eta\ne 0$,
  then the singular principal curvatures are all negative at $p$.
\end{enumerate}
\end{theorem}
When $m=2$, the assertion has been proved in \cite{SUY1}.
The first assertion of \cite[Theorem 5.1]{SUY1}
is essentially same statement as \ref{item:bdd:1}.
We shall prove it for general $m$.
\begin{example}
 Consider a front
 \[
     f\colon{}M^3:=S^{2}\times\R\ni (p,t)\longmapsto 
        \begin{pmatrix}
	 (a+t^2)p , t^3
	\end{pmatrix}\in\R^{4},
 \]
 where $S^2:=\{(x,y,z,0)\in \R^4\,;\,x^2+y^2+z^2=1\}$
 and $a$ is a positive constant.
 The singular set of $f$ is $\Sigma:=S^2\times\{0\}$, 
 which consists of $A_2$-points, and $\partial/\partial t$
 gives the null vector field.
 We set $\tilde\eta=\partial_t$, which is the extended null vector
 field.  One can easily see that this front satisfies the condition
\ref{item:bdd:3} of Theorem \ref{thm:bdd},
and
all principal curvatures are equal to $-1/a$.
\end{example}

First, we choose a coordinate system around an 
$A_2$-singular point:
\begin{lemma}\label{lem:coord}
 Let $(M^m,\E,\inner{~}{~},D,\phi,\psi)$ be a frontal bundle
 over an oriented $m$-manifold $M^m$, and let $p\in M^{m}$ be 
 an $A_2$-singular point of $\phi$.
 We fix $X\in T_p\Sigma_{\phi}\setminus\{0\}$.
 Then there exists a local coordinate system $(u_1,\dots,u_m)$
 of $M^m$ on a neighborhood $U$ of $p$ such that
 \begin{enum}
  \item\label{item:coord:1}
       The $\phi$-singular set $\Sigma_{\phi}$ is parametrized as
       \[
	  \Sigma_{\phi}\cap U = \{(u_1,\dots,u_m)\,;\,u_m=0\}.
       \]
  \item\label{item:coord:2}
       $X=\partial_1$ at $p$.
  \item\label{item:coord:3}
       $\partial_m$
       is a null vector field on $\Sigma_{\phi}\cap U$.
  \item\label{item:coord:4}
       For each $j=1,\dots,m-1$,
       $\inner{\phi_j}{D_m\phi_m}=0$ holds at $p$.	
 \end{enum}
 Here, we denote
       \[
           \partial_j = \frac{\partial}{\partial u_j},\ 
	   \phi_j = \phi(\partial_j),\ 
	   \psi_j = \psi(\partial_j),\ 
           \ \text{and}\quad
           D_j = D_{\partial_j}
           \quad (j=1,\dots,m).
       \]
\end{lemma}
\begin{proof}
 Since $p$ is a non-degenerate singular point, the singular set
 $\Sigma_{\phi}$ 
 is a smooth hypersurface on a neighborhood of $p$.
 Moreover, the null vector field is transversal to $\Sigma_\phi$
 because $p$ is an $A_2$-point.
 Then one can choose a coordinate system $(u_1,\dots,u_m)$
 around $p$ such that \ref{item:coord:1}--\ref{item:coord:3}
 hold.

 We take a new coordinate system $(\tilde u_1,\dots,\tilde u_m)$
 as 
 \[
    \left\{
    \begin{array}{ll}
    \tilde u_j &:=u_j + (u_m)^2a_j \qquad (j=1,\dots,m-1),\\
    \tilde u_m &:=u_m,
    \end{array}
    \right.
 \]
 where $a_j$ ($j=1,\dots,m-1$) are constants.
 Then we have
 \[
   \begin{cases}
    \dfrac{\partial}{\partial \tilde u_j} &=
        \dfrac{\partial}{\partial u_j} \qquad (j=1,\dots,m-1), \\[6pt]
    \dfrac{\partial}{\partial \tilde u_m} &=
        -2 u_m\left(
        \displaystyle\sum_{j=1}^{m-1}
        a_j\dfrac{\partial}{\partial u_j}\right)
        + \dfrac{\partial}{\partial u_m},
   \end{cases}
 \]
 and thus
 \[
      \phi\left(\frac{\partial}{\partial \tilde u_m}\right)=
      \phi\left(
      \frac{\partial}{\partial u_m}\right)     
          -2
        u_m\sum_{j=1}^{m-1} a_j\phi\left(\frac{\partial}{\partial u_j}
      \right).
 \]
 Since $\partial/\partial u_m=\partial/\partial \tilde u_m$ at $p$, we
 have that
 \[
     D_{\partial/\partial \tilde u_m}^{}
     \phi\left(\frac{\partial}{\partial \tilde u_m}\right)=
     D_{m}
     \phi_m
     -2 
        \sum_{j=1}^{m-1} a_j\phi_j
 \]
 at $p$.
 If we set $h_{ij}:=\inner{\phi_i}{\phi_j}$, then 
 \ref{item:coord:4} is equivalent to the equations
 \begin{equation}\label{eq:a-eq}
    2\sum_{j=1}^{m-1} a_jh_{jk}=\inner{ D_{m}\phi_m}{\phi_k}
                \qquad (k=1,2,\dots,m-1).
 \end{equation}
 Since $(h_{jk})_{j,k=1,\dots,m-1}$ is a non-singular matrix,
 we can choose $a_1,\dots,a_{m-1}$ so that \eqref{eq:a-eq} holds,
 and $(\tilde u_1,\dots,\tilde u_m)$
 satisfies \ref{item:coord:1}--\ref{item:coord:4}.
\end{proof}
\begin{corollary}\label{cor:coord2}
 Let $(u_1,\dots,u_m)$ be a coordinate system as in
 Lemma~\ref{lem:coord}, 
 and assume $(M^m,\E,\inner{~}{~},D,\phi,\psi)$ a front bundle.
 Then both $D_m\phi_m$ and $\psi_m$ are
 non-zero vectors perpendicular to $\phi_j$ {\rm(}$j=1,\dots,m-1${\rm)}
 at $p$.
 In particular,
 $D_m\phi_m$ is proportional to $\psi_m$ at $p$.
\end{corollary}
\begin{proof}
 By \ref{item:coord:3} of Lemma~\ref{lem:coord}, 
 $\phi_m=0$ holds on $\Sigma_{\phi}$.
 Since $p\in\Sigma_\phi$ 
 is a non-degenerate singular point, $d\lambda_\phi(p)\neq 0$.
 Then by \ref{item:coord:1}, it holds that 
 $\partial_m\lambda_\phi(p)\neq 0$:
 \[
       \partial_m\lambda_\phi
       = \partial_m\mu(\phi_1,\dots,\phi_m)
       = \mu(\phi_1,\dots,\phi_{m-1},D_m\phi_m)\neq 0\qquad \text{(at $p$)}.
 \]
 Hence $\{\phi_1,\dots,\phi_{m-1},D_m\phi_m\}$
 is linearly independent at $p$.
 That is, $D_m\phi_m$ is a non-zero vector 
 which is perpendicular to $\{\phi_1,\dots,\phi_{m-1}\}$ at $p$.

 On the other hand, by \eqref{eq:compati}, we have
 \[
    \inner{\phi_j}{\psi_m} = \inner{\psi_j}{\phi_m} = 0
    \qquad (j=1,\dots,m-1)
 \]
 on $\Sigma_\phi$.
 Thus $\psi_m(p)$ is perpendicular to
 $\{\phi_1(p),\dots,\phi_{m-1}(p)\}$,
 that is, proportional to $D_m\phi_m$ at $p$.
 Here, by \eqref{eq:front}, $\phi_m(p)=0$ implies $\psi_m(p)\neq 0$.
 Thus we have the conclusion.
\end{proof}

\begin{proof}[%
 Proof of \ref{item:bdd:1} and \ref{item:bdd:2} of Theorem~\ref{thm:bdd}]
 Let  $(u_1,\dots,u_m)$ be 
 a local coordinate system  on a neighborhood of $p$
 as in Lemma~\ref{lem:coord},
 and  set
 \begin{align*}
  h_1 &:=\inner{\phi_1}{\phi_1}\inner{\phi_m}{\phi_m}-
         \inner{\phi_1}{\phi_m}^2,\\
  h_2 &:=\inner{\phi_1}{\psi_1}\inner{\phi_m}{\psi_m}-
         \inner{\phi_m}{\psi_1}^2
 \end{align*}
 on a neighborhood of $p$.
 Then $K^{\ext}(\partial_1\wedge\partial_m)=h_2/h_1$
 on $U\setminus\Sigma_{\phi}$.
 Since $\phi_m=0$ on the $\phi$-singular set 
 $U\setminus\Sigma_{\phi}=\{u_m=0\}$,
 \[
     h_1 =0,\qquad
     \frac{\partial h_1}{\partial u_m}=0,\qquad
     h_2=0,
 \]
 whenever $u_m=0$.
 Then there exist smooth functions $\tilde h_1$ and $\tilde h_2$ 
 on a neighborhood of $p$ such that
 \[
      h_1=(u_m) ^2 \tilde h_1 \qquad\text{and}\qquad h_2 = u_m \tilde h_2.
 \]
 Since $\phi_m=0$ on $\{u_m=0\}$, and
 since  $\{\phi_1,D_m\phi_m\}$  are linearly independent,
 as seen in the proof of Corollary~\ref{cor:coord2},
 \begin{align*}
    \tilde h_1|_{u_m=0}
    &=
    \frac{1}{2}
    \left.\frac{\partial^2}{\partial u_m{}^2}\right|_{u_m=0} h_1\\
    &=
    \frac{1}{2}\left(
     \inner{\phi_1}{\phi_1}\inner{D_m\phi_m}{D_m\phi_m}
    -\inner{\phi_1}{D_m\phi_m}^2\right)\\
    &=\frac{1}{2}|\phi_1\wedge D_m\phi_m|^2 >0
 \end{align*} 
 holds on the singular set near $p$.
 On the other hand, we have
 \[
    \tilde h_2|_{u_m=0}
    =
    \left.\frac{\partial}{\partial u_m}\right|_{u_m=0} 
    h_2=
    \inner{\phi_1}{\psi_1}\inner{D_m\phi_m}{\psi_m}.
 \]
 We assume that $K^{\ext}(\partial_1\wedge\partial_m)$ is 
 bounded on $U\setminus\Sigma_{\phi}$.
 Then $h_2/h_1=\tilde h_2/(u_m\tilde h_1)$ is bounded
 on $U\setminus\Sigma_{\phi}$.
 Thus, $\tilde h_2$ must vanish on the singular set near $p$.
 Here, 
 $\inner{D_m\phi_m}{\psi_m}\neq 0$ holds on a neighborhood of $p$
 because of  Corollary~\ref{cor:coord2}.
 Thus we have
 $\inner{\phi_1}{\psi_1}=-\second(X,X)=0$
 on a singular set near $p$.
 Here, since $\second(\partial_m,\partial_m)=-\inner{\phi_m}{\psi_m}=0$ 
 and $X$ is an arbitrary vector on $T_p\Sigma_{\phi}$,
 $\second(Y,Y)=0$ holds for all $Y\in T_pM^m$.
 Since $\second$ is a symmetric $2$-tensor, we have 
 \ref{item:bdd:1}.

 On the other hand, if $K^{\ext}$ is unbounded on
 $U\setminus\Sigma_{\phi}$, 
 the function $\tilde h_2$ does not vanish at $p$.
 Then 
 $K^{\ext}(\partial_1\wedge\partial_m) =  (1/u_m)(\tilde h_2/\tilde h_1)$
 changes sign at $\Sigma_{\phi}$.
 This implies \ref{item:bdd:2}.
\end{proof}
\begin{proof}[Proof of \ref{item:bdd:3} of Theorem~\ref{thm:bdd}]
 We use the same notations as in the proof of the first part.
 Then it holds that 
 \begin{align*}
  \left.\frac{\partial \tilde h_2}{\partial u_m}\right|_{u_m=0} &=
  \left.\frac{\partial^2}{\partial u_m{}^2}\right|_{u_m=0} 
  \left(
    \inner{\phi_1}{\psi_1}\inner{\phi_m}{\psi_m} -
    \inner{\phi_m}{\psi_1}^2
  \right)
  \\
  &=
    (\partial_m \inner{\phi_1}{\psi_1})(\partial_m\inner{\phi_m}{\psi_m})
    -\left(\partial_m\inner{\phi_m}{\psi_1}\right)^2,
 \end{align*}
 because
 $\phi_m=0$ and $\inner{\phi_1}{\psi_m}=\inner{\phi_m}{\psi_m}=0$
 on the singular set.
 Thus, 
 \begin{equation}\label{eq:limit-k-ext}
    \lim_{\scriptsize{
      \begin{array}{l}q\to p\\q\not\in\Sigma_{\phi}\end{array}}}
    K^{\ext}(q)(\partial_1\wedge\partial_m)
    =
    \frac{
    \partial_m \inner{\phi_1}{\psi_1}\partial_m\inner{\phi_m}{\psi_m}
    -\left(\partial_m\inner{\phi_m}{\psi_1}\right)^2
    }{%
       |\phi_1\wedge D_m\phi_m|^2 
    }.
 \end{equation}
 Here, the assumption of the theorem implies that
 the value \eqref{eq:limit-k-ext} is greater than or equal to $\delta$.
 We consider the case that $\delta>0$.
 Then it holds that
 \begin{equation}\label{eq:k-positive}
    (\partial_m
     \inner{\phi_1}{\psi_1})(\partial_m\inner{\phi_m}{\psi_m})> 0
     \qquad\text{at $p$}
 \end{equation}
 because of \eqref{eq:limit-k-ext}.
 (If $\delta=0$, then the left-hand side of \eqref{eq:k-positive}
 is non-negative.)
 Since $\phi_m=0$ and $\inner{\phi_1}{\psi_m}=0$
 on the singular set $\Sigma_{\phi}$,
 we have
 \begin{align*}
    \partial_m\inner{\phi_1}{\psi_1}&=
    \inner{D_m\phi_1}{\psi_1} + \inner{\phi_1}{D_m\psi_1}
    =
    \inner{D_1\phi_m}{\psi_1} + \inner{\phi_1}{D_1\psi_m}\\
  &=
    \partial_1\inner{\phi_m}{\psi_1} -
    \inner{\phi_m}{D_m\psi_1}+ 
    \partial_1\inner{\phi_1}{\psi_m}-
    \inner{D_1\phi_1}{\psi_m}\\
  &=-\inner{D_1\phi_1}{\psi_m}
 \end{align*}
 at $p$.
 Since $D_m\phi_m$ is proportional to  $\psi_m$ by
 Corollary~\ref{cor:coord2},
 this is written as
 \begin{equation}\label{eq:h11m}
  \begin{aligned}
    \partial_m\inner{\phi_1}{\psi_1}&
    =-\inner{D_1\phi_1}{\psi_m}\\
    &=-\frac{\inner{D_1\phi_1}{D_m\phi_m}\inner{D_m\phi_m}{\psi_m}}
           {|D_m\phi_m|^2}
    \qquad\text{at $p$}.
  \end{aligned}
 \end{equation}
 On the other hand, 
 \begin{equation}\label{eq:hmmm}
  \partial_m\inner{\phi_m}{\psi_m}=\inner{D_m\phi_m}{\psi_m}
 \end{equation}
 holds at $p$.
 By \eqref{eq:k-positive}, \eqref{eq:h11m} and \eqref{eq:hmmm},
 we have
 \begin{equation}\label{eq:k-pos2}
    \inner{D_1\phi_1}{D_m\phi_m}<0
 \end{equation}
 at $p$.
 Next, we compute the $\phi$-singular normal curvature
 $\kappa_{\phi}(\partial_1)$ with respect to the direction 
 $\partial_1$ at $p$.
 Let
 \[
    \vect{n}=\frac{\phi_1\wedge\dots\wedge\phi_{m-1}}{%
                   |\phi_1\wedge\dots\wedge\phi_{m-1}|},
 \]
 which is the unit conormal vector field such that
 $\{\phi_1,\dots,\phi_{m-1},\vect{n}\}$ is positively
 oriented.
 Then
 \[
   \kappa_{\phi}(\partial_1)
   = -\epsilon\frac{\inner{D_1\vect{n}}{\phi_1}}{\inner{\phi_1}{\phi_1}}
   = \epsilon\frac{\inner{\vect{n}}{D_1\phi_1}}{\inner{\phi_1}{\phi_1}},
 \]
 where
 \begin{align*}
     \epsilon &= \sgn(\partial_m\lambda_\phi)
       = \sgn\left(\partial_m\mu(\phi_1,\dots,\phi_m)\right)= 
     \sgn \mu(\phi_1,\dots,D_m\phi_m) \\
      &=\sgn\inner{\phi_1\wedge\dots\wedge\phi_{m-1}}{D_m\phi_m}
             =\sgn\inner{\vect{n}}{D_m\phi_m}.
 \end{align*}
 Here, by Corollary~\ref{cor:coord2}, $D_m\phi_m$ is perpendicular
 to $\{\phi_1,\dots,\phi_{m-1}\}$, that is, 
 it is proportional to $\vect{n}$.
 Thus, \eqref{eq:k-pos2} yields
 \[
   \sgn\left(\kappa_{\phi}(\partial_1)\right) = 
   \sgn(\inner{D_m\phi_m}{\vect{n}}\inner{D_1\phi_1}{\vect{n}})=
   \sgn\inner{D_m\phi_m}{D_1\phi_1}< 0
 \]
 at $p$.
 (When $\delta=0$, $\kappa_{\phi}(\partial_1)$
 is  non-positive.)
 Hence we have the conclusion.
\end{proof}
\begin{proof}[Proof of Corollary \ref{cor:hyper}]
 For a front bundle induced by a front in $\R^{m+1}$ 
 (see Example~\ref{ex:front-bundle}),
 the sectional curvature of the singular set
 spanned by two singular principal directions
 is equal to the product of the two singular 
 principal curvatures by the Gauss equation \eqref{eq:G}.
 Thus, we have Corollary \ref{cor:hyper}
 in the introduction.
\end{proof}

\section*{Acknowledgement}
 The authors thank Wayne Rossman
 for careful reading of the first draft for giving valuable comments.

\end{document}